\documentclass[12pt]{amsart}

\usepackage{amssymb,amsmath}
\usepackage{amssymb,latexsym}
\usepackage{amsfonts}
\usepackage{amssymb}
\usepackage{stmaryrd}
\usepackage{longtable}
\usepackage{amsmath, geometry, amssymb} 
\usepackage{chngcntr}
\counterwithin{table}{section}
\numberwithin{equation}{section}
\newtheorem{theorem}{Theorem}[section]

\newtheorem{corollary}{Corollary}[section]

\newcommand{\sqr}[2]{{\vcenter{\vbox{\hrule height#2pt
                \hbox{\vrule width#2pt height#1pt \kern#1pt
                \vrule width#2pt}\hrule height#2pt}}}}

\newcommand{\beq}{\begin{equation}}
\newcommand{\eeq}{\end{equation}}
\newcommand{\beqar}{\begin{eqnarray}}
\newcommand{\eeqar}{\end{eqnarray}}
\def\beqars{\begin{eqnarray*}}
\def\eeqars{\end{eqnarray*}}

\newcommand{\smod}[1]{\hspace{-1mm} \pmod{#1}}

\def \ds{\displaystyle}

\newcommand{\nn}{\mathbb{N}}
\newcommand{\zz}{\mathbb{Z}}

\newcommand{\cc}{\mathbb{C}}

\newcommand{\hh}{\mathbb{H}}

\newcommand{\kksum}[2]{\scriptsize
            \sum_{
            \begin{array}{c}
             #1\\
             #2
             \end{array}}}

\allowdisplaybreaks

\begin{document}

\title{The convolution sum 
$\displaystyle \mathbf{\sum_{al+bm=n} \hspace{-4mm} \sigma(l) \sigma(m)}$ 
for $\mathbf{(a,b)=(1,28), (4,7), (1,14), (2,7), (1,7)}$ 
}  

\author{Ay\c{s}e Alaca, \c{S}aban Alaca, Eb\'{e}n\'{e}zer Ntienjem}

\maketitle

\markboth{AY\c{S}E ALACA, \c{S}ABAN ALACA, EB\'{E}N\'{E}ZER NTIENJEM}
{THE CONVOLUTION SUM $W_{a,b} (n)$ FOR $(a,b)=(1,28), (4,7), (1,14), (2,7), (1,7)$}


\begin{abstract}
We evaluate the convolution sum
$\displaystyle   W_{a,b}(n):= \sum_{al+bm=n} \hspace{-3mm} \sigma(l) \sigma(m)$ 
for $(a,b)=(1,28), (4,7), (2,7)$ for all positive integers $n$. We use a modular form approach. 
We also re-evaluate the known sums $W_{1,14}(n)$ and $W_{1,7}(n)$ with our method. 
We then use these evaluations to determine the number of representations of $n$ by the octonary quadratic form
$x_1^2 + x_2^2 +x_3^2  + x_4^2 + 7(x_5^2 + x_6^2 + x_7^2 +  x_8^2)$.
Finally we compare our evaluations of the sums $W_{1,7}(n)$ and  $W_{1,14}(n)$ with the evaluations of Lemire and Williams  \cite{lemire} 
and  Royer \cite{royer} to express the modular forms $\Delta_{4,7}(z)$, $\Delta_{4,14, 1}(z)$ and $\Delta_{4,14, 2}(z)$ 
(given in  \cite{lemire, royer}) as linear combinations of eta quotients.  

\vspace{2mm}


\vspace{2mm}

\end{abstract}

\section{Introduction}

Let $\nn$, $\nn_0$, $\zz$  and $\cc$
denote the sets of positive integers, nonnegative integers, integers and  complex numbers respectively.
For $k, n \in \nn$ the sum of divisors function $\sigma_k (n)$ is defined by
$\ds 
\sigma_k(n) = \sum_{ d \mid n } d^k$,
where $d$ runs through the positive divisors of $n$.
If $n\notin \nn$ we set $\sigma_k (n) =0$.
We write $\sigma (n)$ for $\sigma_1 (n)$.
For $a,b \in \nn$  with $a \leq b$ we define the convolution sum
$W_{a,b}(n)$~by
\begin{align}
W_{a,b}(n):= \hspace{-3mm} \kksum{(l,m)\in \nn^2}{al+bm=n} \hspace{-3mm} \sigma(l)\sigma(m).
\end{align}
Set $g = \gcd (a,b)$. Clearly 
\begin{align*}
W_{a,b}(n) = \begin{cases} 
W_{a/g, b/g} (n/g) &\text{if } g \mid n,  \\
0 &\text{if } g \nmid n .
\end{cases}
\end{align*}
Hence we may suppose that $\gcd (a,b) =1$. 
The convolution sum $W_{a,b} (n)$ has been evaluated for
\begin{align*}
(a,b)=&(1,b) \text{ for } 1\leq b\leq 16, 18, 20, 23, 24, 25, 27, 32, 36, \\
& (2,3),  (2,5), (2,9), (3,4), (3,5), (3,8), (4,5), (4,9). 
\end{align*} 
See, for example, 
\cite{convo27, ChanCooper, CooperToh, CooperYe,  Huard, lemire,  ramak, royer,  Xia, Ye}.

In this paper we evaluate the convolution sum $W_{a,b}(n)$ for 
\begin{align*}
(a,b)=(1,28), (4,7), (2,7).
\end{align*}
We use a modular form approach. 
The  sum $W_{1,14}(n)$ has been evaluated by Royer \cite{royer}, and   
the  sum $W_{1,7}(n)$ has been evaluated by Lemire and Williams  \cite{lemire} and  later by Cooper and Toh \cite{CooperToh}. 
We re-evaluate the  sums $W_{1,14}(n)$ and $W_{1,7}(n)$ with our method. 
Our results for the sums $W_{1,14}(n)$ and $W_{1,7}(n)$ agree with those in \cite{royer} and \cite{lemire, CooperToh}, respectively. 

For $l,n\in \nn$ let $ R_{l}(n)$ denote the number of representations of $n$ by the octonary quadratic form
$\displaystyle   x_1^2 + x_2^2  + x_3^2 + x_4^2 + l(x_5^2 + x_6^2 + x_7^2 + x_8^2)$, namely
\begin{align}
&R_{l}(n):={\rm card}\{(x_1,x_2,x_3,x_4,x_5,x_6,x_7,x_8)\in \zz^8 \mid   \nonumber \\
&\hspace{45mm} 
n = x_1^2 + x_2^2 + x_3^2 + x_4^2+ l(x_5^2 + x_6^2 + x_7^2 + x_8^2) \}.
\end{align}
Explicit formulas for $ R_{l}(n)$ are known for $l=1,2,3,4,5,6,8$, see, for example,  \cite{octo1236, lomadze,  CooperYe, ramak, convo27}.  
We  use  the evaluations of the convolution  sums $W_{1,28}(n)$,  $W_{4,7}(n)$ and $W_{1,7}(n)$ 
to determine an explicit formula for $R_7 (n)$. 

Finally we compare our evaluations of the sums $W_{1,7}(n)$ and  $W_{1,14}(n)$ with the evaluations of Lemire and Williams  \cite{lemire} 
and  Royer \cite{royer} to express the modular forms $\Delta_{4,7}(z)$, $\Delta_{4,14, 1}(z)$ and $\Delta_{4,14, 2}(z)$ 
(given in  \cite{lemire, royer}) as linear combinations of eta quotients. 

\section{Preliminary results}

Let $N\in\nn$ and $k \in \zz$. Let $\Gamma_0(N)$ be the modular subgroup defined by
\beqars
\Gamma_0(N) = \left\{ \left(
\begin{array}{cc}
a & b \\
c & d
\end{array}
\right)  \Big | \,  a,b,c,d\in \zz ,~ ad-bc = 1,~c \equiv 0 \smod {N}
\right\} .
\eeqars 
We write $M_k(\Gamma_0(N))$ to denote the space of modular forms of weight $k$ and level~$N$. 
It is known  (see for example \cite[p. 83]{stein}) that
\beqar
M_k (\Gamma_0(N)) = E_k (\Gamma_0(N)) \oplus S_k(\Gamma_0(N)),
\eeqar
where $E_k (\Gamma_0(N))$ and $S_k(\Gamma_0(N))$ are the corresponding subspaces of Eisenstein forms and cusp forms of weight $k$
with trivial multiplier system for the modular subgroup $\Gamma_0(N)$.

The Dedekind eta function $\eta (z)$ is the holomorphic function defined on the upper half plane $\hh = \{ z \in \cc \mid \mbox{\rm Im}(z) >0 \}$ 
by the product formula
\begin{align}
\eta (z) = e^{\pi i z/12} \prod_{n=1}^{\infty} (1-e^{2\pi inz}).
\end{align}
We set $q:=q(z)=e^{2 \pi i z}$. 
Then we can express the Dedekind eta function $\eta (z)$  in (2.2) as 
\begin{align}
\eta (z) = q^{1/24} \prod_{n=1}^{\infty} (1- q^n).
\end{align}
A product of the form
\begin{align*}
f(z) = \prod_{1\leq \delta \mid N} \eta^{r_{\delta}} ( \delta z) ,
\end{align*}
where $r_{\delta} \in \zz$, not all zero, is called an eta quotient.  
 We define the following nine eta quotients
\begin{align}
&C_1(q):= \frac{\eta^5(z)\eta^5(7z)}{\eta(2z)\eta(14z)}, \\
&C_2(q):= \eta^2(z)\eta^2(2z)\eta^2(7z)\eta^2(14z), \\
&C_3(q):= \frac{\eta^6(z)\eta^{6}(14z)}{\eta^2(2z)\eta^2(7z)}, \\
&C_4(q):= \frac{\eta^6(2z)\eta^6(7z)}{\eta^{2}(z)\eta^{2}(14z)}, \\
&C_5(q):= \eta^2 (4z) \eta^4(14z) \eta^2(28z), \\
&C_6(q):= \frac{\eta^6 (2z)  \eta^6(28z)}{\eta^2(4z) \eta^2(14z)}, \\
&C_7(q):= \frac{\eta^4 (2z) \eta^6(28z)}{\eta^2(4z)}, \\
&C_8(q):= \frac{\eta (z) \eta(2z) \eta(7z) \eta^8 (28z)}{\eta^3(14z)}, \\
&C_9(q):= \frac{\eta(2z) \eta(4z) \eta^9(28z)}{\eta^3(14z)},
\end{align}
and integers $c_r(n)$ ($n\in\nn$) for $r\in\{1,2,3,4,5,6,7,8,9\}$ by
\begin{align}
C_r(q)=\sum_{n=1}^{\infty} c_r(n) q^n.
\end{align} 

We use  the following theorem to determine if a given eta quotient $f(z)$ is in $\ds M_k (\Gamma_0(N))$.
See  \cite[Theorem 5.7, p. 99]{Kilford} and  \cite[Corollary 2.3, p. 37]{Kohler}.

\begin{theorem}[Ligozat]  
Let $N\in \nn$  and 
$\ds f(z) = \prod_{1 \leq \delta \mid N} \eta^{r_{\delta}}(\delta z)$
be an eta quotient.
Let $\ds k = \frac{1}{2} \sum_{1 \leq \delta \mid N} r_{\delta}$ and $\ds s=  \prod_{1 \leq \delta \mid N} \delta^{r_{\delta}}$. 
Suppose that the following conditions are satisfied: 

{\rm (i)} $\ds \sum_{ 1\leq  \delta \mid N} \delta \cdot r_{\delta} \equiv 0 \smod {24}$,

{\rm (ii)} $\ds \sum_{ 1 \leq \delta \mid N} \frac{N}{\delta} \cdot r_{\delta} \equiv 0 \smod {24}$,

{\rm (iii)} $\ds \sum_{1 \leq \delta \mid N} \frac{ \gcd (d, \delta)^2 \cdot r_{\delta} }{\delta} \geq 0 $ 
for each positive divisor $d$ of  $N$,

{\rm (iv)} $k$ is an even integer,

{\rm (v)} $s$ is the square of a rational number. 

\noindent
Then $f(z)$ is in  $\ds M_k (\Gamma_0(N))$.

{\rm (iii)$'$}
In addition to the above conditions, if the inequality in {\rm (iii)}  is strict for each positive divisor $d$ of $N$, 
then  $f(z)$ is in $S_k(\Gamma_0(N))$. 
\end{theorem}

We note that we have used MAPLE to find the above eta quotients $C_j (q)$ for $1\leq j \leq 9$ in a way that 
they satisfy Theorem 2.1 for $N=28$ and $k=4$.

The Eisenstein series $L(q)$ and $M(q)$ are defined as 
\begin{align}
L(q):=& 1 -24 \sum_{n=1}^{\infty} \sigma (n) q^n , \\
M(q):=& 1+240 \sum_{n=1}^{\infty} \sigma_3(n) q^n,
\end{align}
respectively.  We use Theorem 2.1  and the Eisenstein series $M(q)$  to give a basis 
for the modular space $M_4(\Gamma_0(28))$ in the following theorem. 

\begin{theorem}  
{\em (a)~} $\{ M(q^t) \mid t=1, 2, 4, 7, 14, 28 \}$ is  a basis for $E_4(\Gamma_0(28))$.

{\em (b)~} $\{C_{j}(q)\mid 1\leq k\leq 9 \}$ is a basis for  $S_4(\Gamma_0(28))$.

{\em (c)~} $\{ M(q^t) \mid t=1, 2, 4, 7, 14, 28 \} \cup \{ C_{j}(q) \mid 1\leq j\leq 9\}$ is a basis for $M_4(\Gamma_0(28))$.
\end{theorem}

\begin{proof}
{\rm (a)~} Appealing to \cite[Theorem 3.8, p. 50]{Kilford} or \cite[Proposition 6.1]{stein}, we see that  
$\dim(E_4(\Gamma_0(28)))=6$. 
Then we see from \cite[Theorem 5.9]{stein} that
$\{ M(q^t) \mid t=1, 2, 4, 7, 14, 28 \}$ is  a basis for $E_4(\Gamma_0(28))$.

{\rm (b)~}  It follows from Theorem 2.1 that each $C_{j}(q)$ is in $S_4(\Gamma_0(28))$ for  $1\leq j\leq 9$. 
By \cite[Theorem 3.8, p. 50]{Kilford} or \cite[Proposition 6.1]{stein}, we have 
$\dim (S_4(\Gamma_0(28)))=9$.
One can see that there is no linear relationship among the eta quotients $C_j(q)~(1\leq j \leq 9)$.
Thus, $\{ C_{j}(q) \mid 1\leq j\leq 9\}$ constitute a basis for $S_4(\Gamma_0(28))$.

{\rm (c)~} The assertion follows from (a), (b) and (2.1).
\end{proof}

We use the Sturm bound $S(N)$  to show the equality of two modular forms in the same modular  space.  
The following theorem gives $S(N)$ for $M_4(\Gamma_0(N))$, 
see \cite[Theorem 3.13 and Proposition 2.11]{Kilford} for a general case.

\begin{theorem}  
Let $f(z), ~g(z)\in M_4(\Gamma_0(N))$ with the Fourier series expansions\\
$\ds f(z)=\sum_{n=0}^{\infty} a_n q^{ n }$ and $\ds g(z)=\sum_{n=0}^{\infty} b_n q^{ n }$.
The Sturm bound $S(N)$ for the modular space $M_4(\Gamma_0(N))$ is given by
\beqars
\ds S(N)=\frac{N}{3} \ds \prod_{p|N} \big( 1+1/p \big),  
\eeqars
and so if $\ds a_n=b_n$ for all $ n \leq S(N)$
then $f(z) = g(z)$. 
\end{theorem}

By Theorem 2.3,  the Sturm bound for the modular space 
$M_4(\Gamma_0(N))$ is 
\begin{align}
S(28) = 16.
\end{align} 
Using (2.16) and Theorem 2.2 we prove Theorem 2.4. 
We then use Theorem 2.4 to determine explicit formulas for our convolution sums in the next section.

\begin{theorem} 
We have
\begin{align*}
&\begin{aligned}
{\big( L( q) - 28 L(q^{28}) \big)}^2  =& \frac{118}{125} M(q) - \frac{21}{125} M(q^2) - \frac{112}{125}M(q^4) 
-\frac{343}{125} M(q^7) \\
& -\frac{1029}{125} M(q^{14}) + \frac{92512}{125} M(q^{28}) - \frac{13452}{25} C_1(q)  - \frac{86004}{25} C_2(q)\\
& + 252 C_3(q) + \frac{40188}{25} C_4(q) + \frac{407232}{25}C_5(q) + \frac{68544}{5} C_6(q) \\
& - \frac{52416}{25} C_7(q) 
+ \frac{2327808}{25} C_8 (q) + \frac{2731008}{25} C_9 (q) , 
\end{aligned} \\
&\begin{aligned}
{\big( 4L( q^4) - 7 L(q^{7}) \big)}^2  =& -\frac{7}{125} M(q) - \frac{21}{125} M(q^2) + \frac{1888}{125}M(q^4) 
+ \frac{5782}{125} M(q^7) \\
&-\frac{1029}{125} M(q^{14}) -\frac{5488}{125} M(q^{28}) - \frac{8364}{175} C_1(q) - \frac{5004}{25} C_2(q) \\
&+ 324 C_3(q)  + \frac{10716}{175} C_4(q) 
 - \frac{24768}{25} C_5(q) + \frac{28224}{5} C_6(q) \\
& -\frac{138816}{25} C_7(q) + \frac{676608}{25} C_8 (q) 
+ \frac{273408}{25} C_9 (q) , 
\end{aligned} \\
&\begin{aligned}
{\big( L( q) - 14 L(q^{14}) \big)}^2  =& \frac{111}{125} M(q) -\frac{56}{125} M(q^2) -\frac{686}{125} M(q^7) 
+ \frac{21756}{125} M(q^{14}) \\
& - \frac{4608}{25} C_2(q) + \frac{672}{25} C_3(q)  + \frac{10272}{25} C_4(q), 
\end{aligned} \\
&\begin{aligned}
{\big( 2L( q^2) - 7 L(q^{7}) \big)}^2  =& -\frac{14}{125} M(q) + \frac{444}{125} M(q^2) + \frac{5439}{125} M(q^7) 
- \frac{2744}{125} M(q^{14}) \\
&- \frac{4608}{25} C_2(q) + \frac{10272}{25} C_3(q) + \frac{672}{25} C_4(q) , 
\end{aligned} \\
&\begin{aligned}
{\big( L( q) - 7 L(q^7) \big)}^2  =& \frac{18}{25} M(q)  + \frac{882}{25} M(q^7)  + \frac{576}{5}( C_1(q) + 4 C_2(q)). 
\end{aligned}
\end{align*}
\end{theorem}

\begin{proof}
We  prove only the first and fourth equalities as the remaining three can be proven similarly. 
Let us prove the  first equality. 
By \cite[Theorem 5.8]{stein} we have $ L( q) - 28 L(q^{28}) \in M_2 (\Gamma_0 (28))$, and so 
\begin{align*}
{\big( L( q) - 28 L(q^{28}) \big)}^2 \in M_4 (\Gamma_0 (28))  .
\end{align*}
By Theorem 2.2(c) there exist coefficients  $x_1, x_2, x_4, x_7, x_{14}, x_{28}, y_1, y_2, \ldots, y_9 \in \cc$ such that 
\begin{align}
{\big( L( q) - 28 L(q^{28}) \big)}^2  =& x_1 M(q) + x_2 M(q^2) + x_4M(q^4) + x_7 M(q^7) + x_{14} M(q^{14}) \nonumber  \\
&+ x_{28} M(q^{28}) + \sum_{i=1}^9 y_i C_i(q). 
\end{align}
Appealing to (2.16), we equate the coefficients of $q^n$ for $0\leq n\leq 16$ on both sides of (2.17), and have a system of linear equations with $17$ equations and $15$ unknowns. 
By using MAPLE we solve this system and find the asserted coefficients. 

Let us now prove the fourth equality. We have 
\begin{align}
 2L( q^2 ) - 7 L(q^7) = L(q) - 7 L(q^7) - (L(q) - 2L(q^2)). 
\end{align}
By \cite[Theorem 5.8]{stein}, we have  
\begin{align}
L( q) - 7 L(q^7) \in M_2 (\Gamma_0 (7)) \text{ and } L( q) - 2 L(q^2) \in M_2 (\Gamma_0 (2)).
\end{align}
Thus it follows from (2.18) and (2.19) that $2 L( q^2) - 7 L(q^7) \in M_2 (\Gamma_0 (14))$, and so 
\begin{align*}
{\big( 2L( q^2 ) - 7 L(q^7) \big)}^2 \in M_4 (\Gamma_0 (14)).
\end{align*}
As $M_4 (\Gamma_0 (14)) \subset M_4 (\Gamma_0 (28))$, we have 
$\ds {\big( 2L( q^2) - 7 L(q^{7}) \big)}^2 \in M_4 (\Gamma_0 (28))$. 
Thus
by Theorem 2.2(c) there exist coefficients $x_1, x_2, x_4, x_7, x_{14}, x_{28}, y_1, y_2, \ldots, y_9 \in \cc$ such that 
\begin{align}
{\big( 2L( q^2) - 7 L(q^{7}) \big)}^2  =& x_1 M(q) + x_2 M(q^2) + x_4M(q^4) + x_7 M(q^7) + x_{14} M(q^{14}) \nonumber  \\
&+ x_{28} M(q^{28}) + \sum_{i=1}^9 y_i C_i(q) . 
\end{align}
We equate the coefficients of $q^n$  for $0 \leq n \leq 16$ on both sides of (2.20) to obtain the asserted coefficients. 
Alternatively, one can show that $\{C_{j}(q)\mid 1\leq k\leq 4\}$ is a basis for $S_4(\Gamma_0(14))$, and find the asserted cofficients 
for the formulas of  $W_{2,7}(n)$, $W_{1, 14}(n)$ and $W_{1, 7}(n)$ accordingly. 
\end{proof}

\section{Evaluating the convolution sum $W_{a,b} (n)$ }

We now present explicit formulas for the convolution sum $W_{a,b} (n)$  
for $(a,b)=(1,28), (4,7), (1,14), (2,7), (1,7)$. 
We make use of Theorem 2.4  and the classical identity 
\begin{align}
L^2(q)=1+\sum_{n=1}^{\infty} (240\sigma_3(n)-288 n\sigma(n))q^n, 
\end{align}
see for example \cite{glaisher}.

\begin{theorem}  
Let $n\in\nn$. Then
\begin{align*}
&\begin{aligned}
W_{1,28} (n) =&\frac{1}{2400}\sigma_{3}(n) + \frac{1}{800}\sigma_{3}(\frac{n}{2}) 
  + \frac{1}{150}\sigma_{3}(\frac{n}{4}) + \frac{49}{2400} \sigma_{3}(\frac{n}{7})  \\
  &+ \frac{49}{800}\sigma_{3}(\frac{n}{14})   + \frac{49}{150}\sigma_{3}(\frac{n}{28})  
   + (\frac{1}{24} -\frac{n}{112})\sigma(n) + (\frac{1}{24}-\frac{n}{4})\sigma(\frac{n}{28})   \\
  & + \frac{1121}{67200} c_{1}(n)   + \frac{2389}{22400} c_{2}(n)     - \frac{1}{128} c_{3}(n) 
  - \frac{3349}{67200} c_{4}(n)     - \frac{101}{200} c_{5}(n)  \\
   &- \frac{17}{40} c_{6}(n)    + \frac{13}{200} c_{7}(n)   - \frac{433}{150} c_{8}(n) 
   - \frac{254}{75} c_{9}(n) ,                  \nonumber  
\end{aligned} \\
&\begin{aligned}
W_{4,7} (n) =&  \frac{1}{2400}\sigma_{3}(n) + \frac{1}{800}\sigma_{3}(\frac{n}{2}) 
  + \frac{1}{150}\sigma_{3}(\frac{n}{4})    + \frac{49}{2400}\sigma_{3}(\frac{n}{7}) \\
  &+ \frac{49}{800}\sigma_{3}(\frac{n}{14})    + \frac{49}{150}\sigma_{3}(\frac{n}{28}) 
  + (\frac{1}{24}-\frac{n}{28})\sigma(\frac{n}{4}) + (\frac{1}{24}-\frac{n}{16})\sigma(\frac{n}{7}) \\ 
  & + \frac{697}{470400} c_{1}(n)  + \frac{139}{22400} c_{2}(n)    - \frac{9}{896} c_{3}(n)   - \frac{893}{470400} c_{4}(n) 
    + \frac{43}{1400} c_{5}(n)  \\
    &- \frac{7}{40} c_{6}(n)    + \frac{241}{1400} c_{7}(n) - \frac{881}{1050} c_{8}(n)  
    - \frac{178}{525} c_{9}(n)    ,            \nonumber 
\end{aligned} \\
&\begin{aligned}
W_{1,14} (n) =& \frac{1}{600} \sigma_3(n) + \frac{1}{150} \sigma_3\left(\frac{n}{2}\right) 
      + \frac{49}{600} \sigma_3\left(\frac{n}{7}\right) + \frac{49}{150} \sigma_3\left(\frac{n}{14}\right) 
      +\left(\frac{1}{24}-\frac{n}{56}\right) \sigma(n) \nonumber \\
   & +\left(\frac{1}{24}-\frac{n}{4}\right) \sigma\left(\frac{n}{14}\right) 
       + \frac{2}{175} c_2(n) - \frac{1}{600} c_3(n) - \frac{107}{4200} c_4(n)  , \nonumber  
\end{aligned} \\
&\begin{aligned}
W_{2,7} (n)=& \frac{1}{600} \sigma_3(n) + \frac{1}{150} \sigma_3\left(\frac{n}{2}\right) 
      + \frac{49}{600} \sigma_3\left(\frac{n}{7}\right) + \frac{49}{150} \sigma_3\left(\frac{n}{14}\right)  
+ \left(\frac{1}{24}-\frac{n}{28}\right) \sigma\left(\frac{n}{2}\right) \nonumber \\
    &+\left(\frac{1}{24}-\frac{n}{8}\right) \sigma\left(\frac{n}{7}\right) 
      + \frac{2}{175} c_2(n) - \frac{107}{4200} c_3(n) - \frac{1}{600} c_4(n) , 
\end{aligned} \\
&\begin{aligned}
W_{1,7} (n)=& \frac{1}{120} \sigma_3(n) + \frac{49}{120} \sigma_3\left(\frac{n}{7}\right) 
    +\left(\frac{1}{24}-\frac{n}{28}\right) \sigma (n)  
   +\left(\frac{1}{24}-\frac{n}{4}\right) \sigma\left(\frac{n}{7}\right)  \\
   &   - \frac{1}{70} c_1(n) - \frac{2}{35} c_2(n) .
\end{aligned}
\end{align*}
\end{theorem}

\begin{proof}
We prove the theorem only for the convolution sum  $W_{1,28}(n)$ as the other four sums can be proven similarly. 
Replacing $q$ by $q^{28}$ in (3.1), we obtain
\beqar
L^2(q^{28})=1+\sum_{n=1}^{\infty}
\Big(240\sigma_3\left(\frac{n}{28}\right) - \frac{72}{7} n \sigma\left(\frac{n}{28}\right)\Big)q^n .
\eeqar
We have
\begin{align}
 L(q)L(q^{28}) &= \Big( 1-24 \sum_{n=1}^{\infty} \sigma (n) q^n \Big)
                     \Big( 1-24 \sum_{n=1}^{\infty} \sigma (n) q^{28n} \Big)   \nonumber \\[1mm]
 & =1-24 \sum_{n=1}^{\infty} \sigma (n) q^n
- 24 \sum_{n=1}^{\infty} \sigma\left(\frac{n}{28}\right) q^n 
 + 576 \sum_{n=1}^{\infty} W_{1,28} (n) q^n .
\end{align}
We obtain from (3.1)-(3.3) that 
\begin{align}
{\big( L(q) - 28 L(q^{28}) \big)}^2  =& L^2(q) + 784 L^2(q^{28}) - 56 L(q) L(q^{28})  \nonumber \\
=&729 + \sum_{n=1}^{\infty} \Big( 240 \sigma_3 (n) + 188160 \sigma_3\left(\frac{n}{28}\right) \nonumber \\
&+32256 \Big(\frac{1}{24}-\frac{n}{112}\Big) \sigma(n)  + 32256 \Big(\frac{1}{24}-\frac{n}{4}\Big) \sigma\left(\frac{n}{28}\right)   \\
&- 32256 W_{1,28} (n) \Big) q^n. \nonumber
\end{align}
We equate the coefficients of $q^n$ on the right hand sides of ${\big( L(q) - 28 L(q^{28}) \big)}^2$ 
in (3.4) and the first part of Theorem 2.4, and solve  for $W_{1,28} (n)$ to obtain the asserted formula. 
\end{proof}

\begin{theorem}  
Let $n\in\nn$. Then
\begin{align*}
R_7 (n)=& 8 \sigma (n) - 32 \sigma \Big(\frac{n}{4}\Big)  
+ 8 \sigma \Big(\frac{n}{7}\Big) - 32 \sigma \Big(\frac{n}{28}\Big) \\
& + 64 W_{(1,7)} (n) + 1024 W_{(1,7)}  \Big(\frac{n}{4}\Big)  -256 \Big( W_{(4,7)}(n) + W_{(1,28)} (n) \Big) .
\end{align*}
\end{theorem}

\begin{proof}
For $n \in \nn_0$ let $r_4(n)$ denote the number of representations of $n$ as sum of four squares, namely
\beqars
&&r_4(n)={\rm card}\{(x_1,x_2,x_3,x_4)\in \zz^4  \mid 
n=x_1^2+x_2^2+x_3^2+x_4^2\},
\eeqars
so that $r_4(0)=1$. It is a classical result of Jacobi, see for example \cite{williams-3},  that
\beqar
r_4(n)=8\kksum{d|n}{4\nmid d} d =8\sigma(n) - 32\sigma\Big(\frac{n}{4}\Big)  \text{ \normalsize for } n\in \nn.
\eeqar
By (1.2) and (3.5) we have
\begin{align}
R_7(n) =& \kksum{(l,m)\in\nn_0^2}{l+7m=n} r_4(l) r_4(m)  \nonumber\\
=& r_4(n) r_4(0) + r_4(0)r_4\left(\frac{n}{7}\right)+  \kksum{(l,m)\in\nn^2}{l+7m=n} r_4(l) r_4(m)  \nonumber \\
=&  8\sigma (n) - 32 \sigma \Big( \frac{n}{4} \Big) + 8\sigma\Big(\frac{n}{7}\Big) - 32 \sigma\Big(\frac{n}{28} \Big)  
+  \kksum{(l,m)\in\nn^2}{l+7m=n} r_4(l) r_4(m)  .
\end{align}
We need to determine the last sum in (3.6). Using (3.5) we obtain
\begin{align}
 \kksum{(l,m)\in\nn^2}{l+7m=n} r_4(l) r_4(m) = &  
\kksum{(l,m)\in\nn^2}{l+7m=n}   \Big( 8\sigma(l) - 32\sigma\Big(\frac{l}{4}\Big)  \Big) 
 \Big( 8\sigma(m) - 32\sigma\Big(\frac{m}{4}\Big)  \Big)  \nonumber \\ 
=& 64 \kksum{(l,m)\in\nn^2}{l+7m=n}  \sigma(l)\sigma (m)  
+1024 \kksum{(l,m)\in\nn^2}{l+7m=n} \sigma\Big(\frac{l}{4}\Big) \sigma\Big(\frac{m}{4}\Big)  \nonumber \\
&-256 \kksum{(l,m)\in\nn^2}{l+7m=n} \sigma\Big(\frac{l}{4}\Big) \sigma(m)   
-256 \kksum{(l,m)\in\nn^2}{l+7m=n} \sigma(l)\sigma\Big(\frac{m}{4}\Big)  \nonumber \\
=& 64 W_{1,7} (n)  + 1024  W_{1,7} (n/4)  -256 \Big(  W_{4,7} (n) +  W_{1,28} (n) \Big) . 
\end{align}
The assertion now follows from (3.6) and (3.7). 
\end{proof}

We deduce the following corollary  from Theorems 3.1 and 3.2.  

\begin{corollary}  
Let $n\in\nn$. Then
\begin{align*}
R_7 (n)=&\frac{8}{25} \sigma_3 (n) - \frac{16}{25} \sigma_3\Big(\frac{n}{2}\Big)  
+ \frac{128}{25} \sigma_3\Big(\frac{n}{4}\Big)  + \frac{392}{25} \sigma_3\Big(\frac{n}{7}\Big)  
- \frac{784}{25} \sigma_3\Big(\frac{n}{14}\Big)  \\
&+ \frac{6272}{25} \sigma_3\Big(\frac{n}{28}\Big)  -\frac{928}{175} c_1(n) - \frac{768}{25} c_2 (n) + \frac{32}{5} c_3 (n) 
+ \frac{2272}{175} c_4(n) \\
&+ \frac{2304}{25} c_5 (n) + \frac{768}{5} c_6(n) - \frac{1152}{25} c_7(n) 
+ \frac{24576}{25} \Big( c_8(n) +  c_9(n)  \Big).
\end{align*}
\end{corollary}

\begin{proof}
We substitute the formulas  of $W_{1,28}(n)$, $W_{4,7}(n)$, 
$W_{1,7}(n)$ and  $W_{1,7}(n/4)$ from Theorem 3.1 into the right hand side of $R_7(n)$ in Theorem 3.2 to obtain 
the formula 
\begin{align}
R_7 (n)=&\frac{8}{25} \sigma_3 (n) - \frac{16}{25} \sigma_3\Big(\frac{n}{2}\Big)  
+ \frac{128}{25} \sigma_3\Big(\frac{n}{4}\Big)  + \frac{392}{25} \sigma_3\Big(\frac{n}{7}\Big)  
- \frac{784}{25} \sigma_3\Big(\frac{n}{14}\Big) \nonumber \\
&+ \frac{6272}{25} \sigma_3\Big(\frac{n}{28}\Big)  
-\frac{6816}{1225} c_1(n) - \frac{5696}{175} c_2 (n) + \frac{32}{7} c_3 (n) 
+ \frac{16224}{1225} c_4(n)   \\
&+ \frac{21248}{175} c_5 (n) + \frac{768}{5} c_6(n) - \frac{10624}{175} c_7(n) 
+ \frac{166912}{175} \Big( c_8(n) + c_9(n) \Big) \nonumber  \\
& -\frac{512}{35} \Big(c_1 (n/4) + 4c_2 (n/4) \Big). \nonumber
\end{align}
By Theorem 2.1, one can see that $C_1(q^4)$ and  $C_2 (q^4)$ are in  $S_4 ( \Gamma_0 (56))$,  
and so we have $C_1(q^4) + C_2 (q^4) \in S_4 ( \Gamma_0 (56))$. 
By Theorem 2.2(b), we know that $C_j (q)~(1 \leq j \leq 9)$  are in $S_4 ( \Gamma_0 (28)) \subset S_4 ( \Gamma_0 (56))$. 
We want to see if there exist coefficients $x_1, x_2, \ldots , x_9 \in \cc$ such that 
\begin{align}
C_1(q^4) + C_2 (q^4) = x_1 C_1(q) + x_2 C_2 (q) + \cdots + x_9 C_9 (q).
\end{align}
By Theorem 2.3, the Sturm bound for the modular space $M_4 ( \Gamma_0 (56))$ is $S(56) = 32$. 
By using MAPLE we equate the coefficients of $q^n$  for $1 \leq n \leq 32$  on both sides of (3.9) 
and have a system of linear equations with 32 equations and nine unknowns. 
We then solve this system to obtain the identity   
\begin{align}
C_1 (q^4) + 4C_2 (q^4) 
=& -\frac{1}{56} C_1 (q) -\frac{1}{8} C_2 (q) - \frac{1}{8} C_3 (q) + \frac{1}{56} C_4 (q)  \\
&+ 2 C_5 (q) - C_7 (q) - 2C_8 (q) - 2C_9 (q) . \nonumber
\end{align}
We deduce from (3.10) that, for $n \in \nn$,  
\begin{align}
c_1(n/4) + 4c_2 (n/4) 
=& -\frac{1}{56} c_1 (n) -\frac{1}{8} c_2 (n) - \frac{1}{8} c_3 (n) + \frac{1}{56} c_4 (n)  \\
&+ 2c_5 (n) - c_7 (n) - 2c_8 (n) - 2c_9 (n) . \nonumber
\end{align}
The asserted expression for $R_7 (n)$  now follows by substituting (3.11) into (3.8). 
\end{proof}

\section{Expressing  $\Delta_{4,7}(z)$, $\ds \Delta_{4,14,1} (z)$ and $\ds \Delta_{4,14,2} (z)$ 
as linear combinations of eta quotients }

In this section we present our observations regarding two results given by Lemire and Williams \cite{lemire} 
and Royer \cite{royer} respectively.  We  express the modular forms $\Delta_{4,7}(z)$, $\ds \Delta_{4,14,1} (z)$ and $\ds \Delta_{4,14,2} (z)$
as linear combinations of eta quotients. 
We  refer the reader to  \cite[Remark 1.2 and Tables 6 and 7]{royer} for the definitions  
of  $\Delta_{4,7}(z)$, $\ds \Delta_{4,14,1} (z)$ and $\ds \Delta_{4,14,2} (z)$. 

We first present our observation for the sum $W_{1,7} (n)$ given in \cite{lemire} and  express 
the cusp form $\Delta_{4,7}(z)$ as a linear combination of two eta quotients. 
The  sum $W_{1,7}(n)$ has been given by Lemire and Williams \cite[Theorem 2]{lemire}  as 
\begin{align}
W_{1,7} (n)=& \frac{1}{120} \sigma_3(n) + \frac{49}{120} \sigma_3\left(\frac{n}{7}\right) 
    +\left(\frac{1}{24}-\frac{n}{28}\right) \sigma (n)   \\
   &  +\left(\frac{1}{24} -\frac{n}{4}\right) \sigma\left(\frac{n}{7}\right) - \frac{1}{70} u(n) ,  \nonumber 
\end{align}
where $u(n)$ is defined by 
\begin{align}
\sum_{n=1}^{\infty} u(n) q^n = q \Big( F^{16}(q) F^8(q^7) + 13q F^{12}(q)   F^{12}(q^7) 
+ 49\, q^2 \, F^8(q)  F^{16}(q^7)  \Big)^{1/3} 
\end{align}
and 
\begin{align}
F(q) =  \prod_{n=1}^{\infty} (1- q^n).
\end{align}
It follows from (2.3) and (4.3) that  
\begin{align*}
F(q)= q^{-1/24} \eta (z). 
\end{align*}
Then (4.2) becomes
\begin{align}
\sum_{n=1}^{\infty} u(n) q^n 
=&  \left(   \eta^{16} (z) \eta^8 (7z) +13 \eta^{12} (z) \eta^{12} (7z) + 49 \eta^8 (z) \eta^{16} (7z)   \right)^{1/3} \\
=&\Delta_{4,7}(z) \text{ (with the notation in \cite[Remark 1.2]{royer})} \nonumber \\
=&\sum_{n=1}^{\infty} \tau_{4,7}(n) q^n . \nonumber 
\end{align}
See also \cite[Corollary 4.1]{CooperToh}. 
Equating the right hand sides of the sum $W_{1,7}(n)$ in Theorem 3.1 and (4.1), we obtain
\begin{align}
u(n) = c_1(n) + 4c_2 (n) \text{ for $n \in \nn$},
\end{align}
where $c_1(n)$ and $c_2(n)$ are given by (2.13).
Thus appealing to (4.5), (4.4), (2.13), (2.5), (2.4) and the Sturm bound given in (2.16), 
we express $\Delta_{4,7}(z)$ as a linear combination of two eta quotients as 
\begin{align}
\Delta_{4,7}(z) &= \sum_{n=1}^{\infty} \tau_{4,7}(n) q^n  =  \sum_{n=1}^{\infty} u(n) q^n = C_1(q) +4 C_2(q) \nonumber \\
 &=  \frac{\eta^5(z)\eta^5(7z)}{\eta(2z)\eta(14z)} + 4\eta^2(z)\eta^2(2z)\eta^2(7z)\eta^2(14z).
\end{align}

We now present our observation for the sum $W_{1,14} (n)$ given in \cite{royer} and  express  
$\ds \Delta_{4,14,1} (z)$ and $\ds \Delta_{4,14,2} (z)$ as linear combinations of eta quotients. 
The  sum $W_{1,14}(n)$ has been given by Royer \cite[Theorem 1.7]{royer} as 
\begin{align}
W_{1,14} (n) =& \frac{1}{600} \sigma_3(n) + \frac{1}{150} \sigma_3\left(\frac{n}{2}\right) 
      + \frac{49}{600} \sigma_3\left(\frac{n}{7}\right) + \frac{49}{150} \sigma_3\left(\frac{n}{14}\right)      \\
   &+\left(\frac{1}{24}-\frac{n}{56}\right) \sigma(n) +\left(\frac{1}{24}-\frac{n}{4}\right) \sigma\left(\frac{n}{14}\right)  \nonumber \\
   & - \frac{3}{350} \tau_{4,7}(n) - \frac{6}{175} \tau_{4,7}(n/2) - \frac{1}{84} \tau_{4,14,1}(n)  - \frac{1}{200} \tau_{4,14,2}(n) , \nonumber  
\end{align}
where  $\tau_{4,7}(n)$ and $\tau_{4,7}(n/2)$ are given by (4.6), and the values of
$ \tau_{4,14,1}(n)$ and $ \tau_{4,14,2}(n)$ for $1\leq n\leq 22$ are given in Tables 6 and 7 of \cite{royer}, respectively. 
Using the values of
$\tau_{4,14,1}(n)$ and $ \tau_{4,14,2}(n)$ in \cite[Tables 6 and 7]{royer} and appealing to 
the Sturm bound given in (2.16) 
we express the modular forms $\ds \Delta_{4,14,1} (z): = \sum_{n=1}^{\infty}  \tau_{4,14,1}(n) q^n$ 
and $\ds \Delta_{4,14,2} (z): = \sum_{n=1}^{\infty}  \tau_{4,14,2}(n) q^n$ as linear combinations of our eta quotients as 
\begin{align*}
\Delta_{4,14,1} (z): &= -C_3(q)+C_4(q)=
- \frac{\eta^6(z)\eta^{6}(14z)}{\eta^2(2z)\eta^2(7z)}+  \frac{\eta^6(2z)\eta^6(7z)}{\eta^{2}(z)\eta^{2}(14z)},\\
\Delta_{4,14,2} (z): &= -4C_2(q)+C_3(q)+C_4(q)\\
&= -4 \eta^2(z)\eta^2(2z)\eta^2(7z)\eta^2(14z)+  \frac{\eta^6(z)\eta^{6}(14z)}{\eta^2(2z)\eta^2(7z)}
+ \frac{\eta^6(2z)\eta^6(7z)}{\eta^{2}(z)\eta^{2}(14z)}.
\end{align*}

\section*{Acknowledgments} 
The authors are grateful to Professor Emeritus Kenneth S. Williams for helpful discussions throughout the course of this research. 
The research of the first two authors was supported 
by Discovery Grants from the Natural Sciences and Engineering Research Council of Canada (RGPIN-418029-2013 and RGPIN-2015-05208).

\vspace{3mm}
\noindent
Centre for Research in Algebra and Number Theory \\
School of Mathematics and Statistics \\
Carleton University \\
Ottawa, Ontario, K1S 5B6, Canada \\

\noindent
AyseAlaca@cunet.carleton.ca\\
SabanAlaca@cunet.carleton.ca\\
Ebenezer.Ntienjem@carleton.ca


\begin{thebibliography}{0}
\bibitem{glaisher}
{J. W. L. Glaisher,}
\newblock On the square of the series in which the coefficients are the sums of the divisors of the exponents,
\newblock  {\it Mess. Math.} {\bf 14} (1884), 156-163.
\end{thebibliography}
\end{document}